\documentclass[12pt,reqno,a4paper]{amsart}
\usepackage{amssymb}
\usepackage{amsfonts}
\usepackage{amsmath, amsthm, amssymb}
\usepackage[english]{babel}
\usepackage[mathcal]{eucal}
\usepackage{fancyhdr}

\title[An Arzel{\`a}-Ascoli theorem]{An Arzel{\`a}-Ascoli theorem for the Hausdorff measure of noncompactness}
\author{Ben Berckmoes}
\address{Departement Wiskunde-Informatica, Middelheimcampus, Middelheimlaan 1, 2020 Antwerp, Belgium}
\email{ben.berckmoes@ua.ac.be}
\date{}

\thanks{The author is PhD fellow at the Fund for Scientific Research of Flanders (FWO)}

\DeclareMathOperator*{\myinf}{in\vphantom{p}f}

\begin{document}

\maketitle

\hyphenation{ex-ten-da-ble}
\hyphenation{Pro-po-si-tion}
\hyphenation{Ge-bie-te}

\newtheorem{pro}{Proposition}[section]
\newtheorem{lem}[pro]{Lemma}
\newtheorem{thm}[pro]{Theorem}
\newtheorem{de}[pro]{Definition}
\newtheorem{co}[pro]{Comment}
\newtheorem{no}[pro]{Notation}
\newtheorem{vb}{Example}
\newtheorem{vbn}[pro]{Examples}
\newtheorem{gev}[pro]{Corollary}

\begin{abstract}
We generalize the Arzel{\`a}-Ascoli theorem in the space of continuous maps on a compact interval with values in Euclidean $N$-space by providing a quantitative link between the Hausdorff measure of noncompactness in this space and a natural measure of non-uniform equicontinuity. The proof hinges upon a classical result of Jung's on the Chebyshev radius.
\end{abstract}

\section{Introduction and statement of the main result}

Fix $N \in \mathbb{N}_0$ and let $\mathcal{C} = \mathcal{C}\left(\left[a,b\right],\mathbb{R}^N\right)$ be the space of continuous $\mathbb{R}^N$-valued maps on the compact interval $\left[a,b\right]$. Let $\left|\cdot\right|$ stand for the Euclidean norm on $\mathbb{R}^N$ and recall that a set $\mathcal{F} \subset \mathcal{C}$ is said to be
\begin{enumerate}
	\item {\em uniformly bounded} iff there exists a universal constant $M > 0$ such that $\left|f(x)\right| \leq M$ for all $f \in \mathcal{F}$ and $x \in \left[a,b\right]$,
	\item {\em uniformly relatively compact} iff each sequence in $\mathcal{F}$ contains a subsequence converging uniformly to a map in $\mathcal{C}$, 
	\item {\em uniformly equicontinuous} iff for each $\epsilon > 0$ there exists $\delta > 0$ such that $\left|f(x) - f(y)\right| < \epsilon$ for all $f \in \mathcal{F}$ and $x,y \in \left[a,b\right]$ with $\left|x - y\right| < \delta$.
	\end{enumerate}
Denote the collection of uniformly bounded sets in $\mathcal{C}$ as $\mathcal{B}_\mathcal{C}$. In this setting the following theorem is a classic (\cite{L93}).

\begin{thm}\label{AA}(Arzel{\`a}-Ascoli)
For $\mathcal{F} \in \mathcal{B}_{\mathcal{C}}$ the following are equivalent:
\begin{enumerate}
	\item $\mathcal{F}$ is uniformly relatively compact.
	\item $\mathcal{F}$ is uniformly equicontinuous.
\end{enumerate}
\end{thm}

Recall that $\mathcal{C}$ is a Banach space under the supremum norm 
\begin{displaymath}
\|f\|_\infty = \sup_{x \in \left[a,b\right]} \left|f(x)\right|
\end{displaymath}
and that for a set $\mathcal{F} \in \mathcal{B}_{\mathcal{C}}$ the {\em Hausdorff measure of noncompactness} (\cite{BG80},\cite{WW96}) is given by
\begin{displaymath}
\mu_{\textrm{\upshape{H}}}(\mathcal{F}) = \myinf_{\mathcal{F}_0} \sup_{f \in \mathcal{F}} \myinf_{g \in \mathcal{F}} \|f - g\|_{\infty},
\end{displaymath}
the first infimum running through all finite sets $\mathcal{F}_0$ in $\mathcal{C}$. It is well known that $\mathcal{F}$ is uniformly relatively compact if and only if $\mu_{\textrm{\upshape{H}}}(\mathcal{F}) = 0$.

For a set $\mathcal{F} \in \mathcal{B}_{\mathcal{C}}$ we define the {\em measure of non-uniform equicontinuity}  as 
\begin{displaymath}
\mu_{\textrm{\upshape{uec}}}(\mathcal{F}) = \myinf_{\delta > 0} \sup_{f \in \mathcal{F}} \sup_{\left|x - y\right| < \delta} \left|f(x) - f(y)\right|,
\end{displaymath}
the second supremum running through all $x,y \in \left[a,b\right]$ with $\left|x - y\right| < \delta$. It is clear that $\mathcal{F}$ is uniformly equicontinuous if and only if $\mu_{\textrm{\upshape{uec}}}(\mathcal{F}) = 0$. 
In \cite{BG80} it was shown that $\mu_{\textrm{\upshape{uec}}}$ is a measure of noncompactness on the space $\mathcal{C}$ (Theorem 11.2).

Theorem \ref{AAQ}, our main result, generalizes Theorem \ref{AA} by linking $\mu_{\textrm{\upshape{H}}}$ and $\mu_{\textrm{\upshape{uec}}}$ quantitatively. The proof is deferred to section 3.

\begin{thm}\label{AAQ}(Arzel{\`a}-Ascoli for the Hausdorff measure of noncompactness)
For $\mathcal{F} \in \mathcal{B}_{\mathcal{C}}$ we have
\begin{displaymath}
\frac{1}{2} \mu_{\textrm{\upshape{uec}}}(\mathcal{F}) \leq \mu_{\textrm{\upshape{H}}}(\mathcal{F}) \leq  \left(\frac{N}{2N + 2}\right)^{1/2} \mu_{\textrm{\upshape{uec}}}(\mathcal{F})
\end{displaymath}
In particular, if $N = 1$, then
\begin{displaymath}
\mu_{\textrm{\upshape{H}}}(\mathcal{F}) = \frac{1}{2} \mu_{\textrm{\upshape{uec}}}(\mathcal{F}).
\end{displaymath}
\end{thm}

\section{A preliminary result of Jung's} 

For a bounded set $A \subset \mathbb{R}^N$, the {\em diameter} is given by
\begin{displaymath}
\textrm{\upshape{diam}}(A) = \sup_{x,y \in A} \left|x-y\right|
\end{displaymath}
and the {\em Chebyshev radius} by
\begin{displaymath}
r(A) = \myinf_{x \in \mathbb{R}^N} \sup_{y \in A} \left|x - y\right|.
\end{displaymath}
It is well known that for each bounded set $A \subset \mathbb{R}^N$ there exists a unique $x_A \in \mathbb{R}^N$ such that 
\begin{displaymath}
\sup_{y \in A} \left|x_A - y\right| = r(A).
\end{displaymath}
The point $x_A$ is called the {\em Chebyshev center of $A$}. A good exposition of the previous notions in a general normed vector space can be found in \cite{H72}, section 33.

Theorem \ref{J} provides a relation between the diameter and the Chebyshev radius of a bounded set in $\mathbb{R}^N$. A beautiful proof can be found in \cite{BW41}. For extensions of the theorem we refer to \cite{A85}, \cite{AFS00}, \cite{R02} and \cite{NN06}.

\begin{thm}\label{J}(Jung)
For a bounded set $A \subset \mathbb{R}^N$ we have
\begin{displaymath}
\frac{1}{2} \textrm{\upshape{diam}}(A) \leq r(A) \leq \left(\frac{N}{2N + 2}\right)^{1/2} \textrm{\upshape{diam}}(A).
\end{displaymath}
\end{thm}

\section{proof of Theorem \ref{AAQ}}

We first need two simple lemmas on linear interpolation.

For $c_0 \in \mathbb{R}^N$ and $r \in \mathbb{R}_0^+$ we denote the {\em closed ball with center $c_0$ and radius $r$} as $B^\star(c_0,r)$.

\begin{lem}\label{lem:Tech1}
Consider $c_1,c_2 \in \mathbb{R}^N$ and $r \in \mathbb{R}^+_0$ and assume that $B^\star(c_1,r) \cap B^\star(c_2,r) \neq \emptyset.$ Let $L$ be the $\mathbb{R}^N$-valued map on the compact interval $\left[\alpha,\beta\right]$ defined by 
\begin{displaymath}
L(x) = \frac{\beta - x}{\beta - \alpha} c_1 + \frac{x - \alpha}{\beta - \alpha} c_2.
\end{displaymath}
Then, for all $x \in \left[\alpha,\beta\right]$ and $y \in B^\star(c_1,r) \cap B^\star(c_2,r)$,
\begin{displaymath}
\left|L(x) - y\right| \leq r.
\end{displaymath}
\end{lem}

\begin{proof}
The calculation
\begin{eqnarray*}
\left|L(x) - y\right| &=& \left|\frac{\beta - x}{\beta - \alpha} (c_1 - y) + \frac{x - \alpha}{\beta - \alpha} (c_2 - y)\right|\\
&\leq& \frac{\beta - x}{\beta - \alpha} \left|c_1 - y\right| + \frac{x - \alpha}{ \beta - \alpha} \left|c_2 - y\right|\\
&\leq& \frac{\beta - x}{\beta - \alpha} r + \frac{x - \alpha}{\beta - \alpha} r\\
&=& r
\end{eqnarray*}
proves the lemma.
\end{proof}

\begin{lem}\label{lem:Tech2}
Consider $c_1,c_2,y_1,y_2 \in \mathbb{R}^N$ and $\epsilon > 0$ and suppose that $\left|c_1 - y_1\right| \leq \epsilon$ and $\left|c_2 - y_2\right| \leq \epsilon$. Let $L$ and $M$ be the $\mathbb{R}^N$-valued maps on the compact interval $\left[\alpha,\beta\right]$ defined by
\begin{displaymath}
L(x) = \frac{\beta - x}{\beta - \alpha} c_1 + \frac{x - \alpha}{\beta - \alpha} c_2
\end{displaymath}
and
\begin{displaymath}
M(x) = \frac{\beta - x}{\beta - \alpha} y_1 + \frac{x - \alpha}{\beta - \alpha} y_2.
\end{displaymath}
Then 
\begin{displaymath}
\|L - M\|_\infty \leq \epsilon.
\end{displaymath}
\end{lem}

\begin{proof}
The calculation
\begin{eqnarray*}
\left|L(x) - M(x)\right| &=& \left|\frac{\beta - x}{\beta - \alpha} (c_1 - y_1) + \frac{x - \alpha}{\beta - \alpha} (c_2 - y_2)\right|\\
&\leq& \frac{\beta - x}{\beta - \alpha} \left|c_1 - y_1\right| + \frac{x - \alpha}{ \beta - \alpha} \left|c_2 - y_2\right|\\
&\leq& \frac{\beta - x}{\beta - \alpha} \epsilon + \frac{x - \alpha}{\beta - \alpha} \epsilon\\
&=& \epsilon
\end{eqnarray*}
proves the lemma.
\end{proof}

\begin{proof}(of Theorem \ref{AAQ})
Let $\mathcal{F} \in \mathcal{B}_{\mathcal{C}}$.

We first show that
\begin{displaymath}
\mu_{\textrm{\upshape{H}}}(\mathcal{F}) \leq \left(\frac{N}{2N + 2}\right)^{1/2} \mu_{\textrm{\upshape{uec}}}(\mathcal{F}).
\end{displaymath}
Fix $\epsilon > 0$. Then, $\mathcal{F}$ being uniformly bounded, we can take a constant $M> 0$ such that 
\begin{eqnarray}
\forall f \in \mathcal{F}, \forall x \in \left[a,b\right] : \left|f(x)\right| \leq M.\label{bdd}
\end{eqnarray}
Pick a finite set $Y \subset \mathbb{R}^N$ for which
\begin{eqnarray}
\forall z \in B^\star(0,3M), \exists y \in Y : \left|y - z\right| \leq \epsilon.\label{YDense}
\end{eqnarray}
Now let $0 < \alpha \leq 2M$ be so that $\mu_{\textrm{\upshape{uec}}}(\mathcal{F}) < \alpha$, i.e. there exists $\delta > 0$ for which 
\begin{eqnarray}
\forall f \in \mathcal{F}, \forall x,y \in \left[a,b\right] : \left|x - y\right| <  \delta \Rightarrow \left|f(x) - f(y)\right| \leq \alpha.\label{eq:eqct}
\end{eqnarray}
Then choose points
\begin{displaymath}
a = x_0 < x_1 < \ldots < x_{2n} < x_{2n + 1} = b,
\end{displaymath}
put
\begin{eqnarray*}
I_0 &=& \left[0,x_2\right[,\\
I_k &=& \left]x_{2k - 1}, x_{2k + 2}\right[ \textrm{ if } k \in \left\{1,\ldots,n-1\right\},\\
I_n &=& \left]x_{2n - 1}, x_{2n + 1}\right]
\end{eqnarray*}
and assume that we have made this choice such that
\begin{eqnarray}
\forall k \in \left\{0, \ldots, n\right\} : \textrm{\upshape{diam}}(I_k) < \delta.\label{eq:diam}
\end{eqnarray}
Furthermore, for each $(y_0, \ldots, y_{2n+ 1}) \in Y^{2n+2}$, let  $L_{(y_0, \ldots, y_{2n+1})}$ be the $\mathbb{R}^N$-valued map on $\left[a,b\right]$ defined by
\begin{eqnarray*}
L_{(y_0, \ldots, y_{2n+1})}(x) = \left\{\begin{array}{clrrrrrrrrrrrrrr}      
\frac{x_1 - x}{x_1 - x_0} y_0 + \frac{x - x_0}{x_1 - x_0} y_1 &\textrm{if}&  x \in \left[x_0,x_1\right]\\       
\vdots\\
\frac{x_{k+1} - x}{x_{k+1} - x_k} y_k + \frac{x - x_k}{x_{k+1} - x_k} y_{k + 1} & \textrm{if}& x \in \left[x_{k},x_{k+1}\right]\\
\vdots\\
\frac{x_{2n+1} - x}{x_{2n+1} - x_{2n}} y_{2n} + \frac{x - x_{2n}}{x_{2n+1} - x_{2n}} y_{2n + 1} &\textrm{if}& x \in \left[x_{2n},x_{2n + 1}\right]
\end{array}\right.
\end{eqnarray*}
and put
\begin{displaymath}
\mathcal{F}_0 = \left\{L_{(y_0,\ldots,y_{2n+1})} \mid (y_0, \ldots, y_{2n + 1}) \in Y^{2n + 2}\right\}.
\end{displaymath}
Then $\mathcal{F}_0$ is a finite subset of $\mathcal{C}$. Now fix $f \in \mathcal{F}$ and let $c_{f,k}$ stand for the Chebyshev center of $f(I_k)$ for each $k \in \left\{0,\ldots, n\right\}$. It follows from (\ref{eq:eqct}) and (\ref{eq:diam}) that $\textrm{\upshape{diam}}f(I_k) \leq \alpha$ and thus, by Theorem \ref{J}, 
\begin{eqnarray}
\forall k \in \left\{0,\ldots,n\right\} : \sup_{x \in I_k} \left|c_{f,k} - f(x)\right| \leq \left(\frac{N}{2N + 2}\right)^{1/2} \alpha.\label{eq:JApp}
\end{eqnarray}
Let $\widetilde{f}$ be the $\mathbb{R}^N$-valued map on $\left[a,b\right]$ defined by
\begin{displaymath}
\widetilde{f}(x) = \left\{\begin{array}{clrrrrrrrrrrrrrr}      
c_{f,0}& \textrm{if}&  x \in \left[x_0,x_1\right]\\       
\frac{x_2 - x}{x_2 - x_1} c_{f,0}+ \frac{x - x_1}{x_2 - x_1} c_{f,1}& \textrm{if}& x \in \left[x_1,x_2\right]\\
c_{f,1} & \textrm{if}& x \in \left[x_2,x_3\right[\\
\frac{x_4 - x}{x_4 - x_3} c_{f,1}+ \frac{x - x_3}{x_4 - x_3} c_{f,2} &\textrm{if}& x \in \left[x_3,x_4\right]\\
\vdots\\
\frac{x_{2k} - x}{x_{2k} - x_{2k-1}} c_{f,k-1}+ \frac{x - x_{2k-1}}{x_{2k} - x_{2k-1}} c_{f,k} & \textrm{if}& x \in \left[x_{2k-1},x_{2k}\right]\\
c_{f,k} & \textrm{if}& x \in \left[x_{2k},x_{2k+1}\right]\\
\frac{x_{2 k + 2} - x}{x_{2 k + 2} - x_{2 k + 1}} c_{f,k}+ \frac{x - x_{2 k + 1}}{x_{2 k + 2} - x_{2 k + 1}} c_{f,k+1} &\textrm{if}& x \in \left[x_{2k+1},x_{2k+2}\right]\\
\vdots\\
\frac{x_{2 n - 2} - x}{x_{2 n - 2} - x_{2 n - 3}} c_{f,n - 2}+ \frac{x - x_{2n - 3}}{x_{2n - 2} - x_{2n - 3}} c_{f,n - 1} &\textrm{if}& x \in \left[x_{2n-3},x_{2n-2}\right]\\
c_{f,n-1} & \textrm{if}& x \in \left[x_{2n-2},x_{2n-1}\right]\\
\frac{x_{2 n} - x}{x_{2 n} - x_{2 n - 1}} c_{f,n-1}+ \frac{x - x_{2 n -1}}{x_{2 n} - x_{2 n - 1}} c_{f,n}  &\textrm{if}& x \in \left[x_{2n-1},x_{2n}\right]\\
c_{f,n} & \textrm{if}& x \in \left[x_{2n},x_{2n+1}\right]
\end{array}\right..
\end{displaymath}
Then (\ref{eq:JApp}) and Lemma \ref{lem:Tech1} learn that 
\begin{eqnarray}
\|\widetilde{f} - f\|_{\infty} \leq \left(\frac{N}{2N + 2}\right)^{1/2} \alpha.\label{ftcf} 
\end{eqnarray} 
Also, it easily follows from (\ref{bdd}) and (\ref{eq:JApp}) that $\|\widetilde{f}\|_{\infty} \leq 3M$ and thus (\ref{YDense}) allows us to choose $\left(y_0, \ldots, y_{2n + 1}\right) \in Y^{2n + 2}$ such that 
\begin{eqnarray}
\forall k \in \left\{0,\ldots, 2n + 1\right\} : \left|y_k - \widetilde{f}(x_k)\right| \leq \epsilon.\label{ykc}
\end{eqnarray}
Combining (\ref{ykc}) and Lemma \ref{lem:Tech2} reveals that
\begin{eqnarray}
\|L_{(y_0,\ldots,y_{2n + 1})} - \widetilde{f}\|_{\infty} \leq \epsilon.\label{ltft}
\end{eqnarray}
But then we have found $L_{(y_0,\ldots,y_{2n + 1})}$ in $\mathcal{F}_0$ for which, by (\ref{ftcf}) and (\ref{ltft}),
\begin{displaymath}
\|L_{(y_0,\ldots,y_{2n+1})} - f\|_\infty \leq \left(\frac{N}{2N + 2}\right)^{1/2} \alpha + \epsilon
\end{displaymath}
which, by the arbitrariness of $\epsilon$, entails that $\mu_{\textrm{\upshape{H}}}(\mathcal{F}) \leq \left(\frac{N}{2N + 2}\right)^{1/2} \alpha$ and thus, by the arbitrariness of $\alpha$, the inequality 
\begin{displaymath}
\mu_{\textrm{\upshape{H}}}(\mathcal{F}) \leq  \left(\frac{N}{2N + 2}\right)^{1/2} \mu_{\textrm{\upshape{uec}}}(\mathcal{F})
\end{displaymath}
is established.

We now prove that 
\begin{displaymath}
\frac{1}{2} \mu_{\textrm{\upshape{uec}}}\left(\mathcal{F}\right) \leq  \mu_{\textrm{\upshape{H}}}\left(\mathcal{F}\right). 
\end{displaymath}
Let $\alpha > 0$ be so that $\mu_{\textrm{\upshape{H}}}\left(\mathcal{F}\right) < \alpha$. Then there exists a finite set $\mathcal{F}_0 \subset \mathcal{C}$ such that for all $f \in \mathcal{F}$ there exists $g \in \mathcal{F}_0$ for which $\|g - f\|_\infty \leq \alpha$. Take $\epsilon > 0$. Since $\mathcal{F}_0$ is uniformly equicontinuous there exists $\delta > 0$ so that 
\begin{eqnarray}
\forall g \in \mathcal{F}_0, \forall x, y \in \left[a,b\right] : \left|x - y\right| < \delta \Rightarrow \left|g(x) - g(y)\right| \leq \epsilon.\label{eq:FzEC}
\end{eqnarray}
Now, for $f \in \mathcal{F}$, choose $g \in \mathcal{F}_0$ such that 
\begin{eqnarray}
\|g - f\|_\infty \leq \alpha.\label{eq:2gCloseTof}
\end{eqnarray}
Then, for $x , y \in \left[a,b\right]$ with $\left|x - y\right| < \delta$, we have, by (\ref{eq:FzEC}) and (\ref{eq:2gCloseTof}),
\begin{displaymath}
\left|f(x) - f(y)\right| \leq \left|f(x) - g(x)\right| + \left|g(x) - g(y)\right| + \left|g(y) - f(y)\right| \leq 2 \alpha + \epsilon
\end{displaymath}
which, by the arbitrariness of $\epsilon$ reveals that  $\mu_{\textrm{\upshape{uec}}}\left(\mathcal{F}\right) \leq 2\alpha$ and thus, by the arbitrariness of $\alpha$, the inequality 
\begin{displaymath}
\frac{1}{2} \mu_{\textrm{\upshape{uec}}}\left(\mathcal{F}\right) \leq \mu_{\textrm{\upshape{H}}}\left(\mathcal{F}\right)
\end{displaymath}
holds.
\end{proof}

\end{document}